\documentclass[11pt]{amsart}
\usepackage[small,width=\linewidth]{caption}

\def\cput(#1,#2)#3{\put(#1,#2){\hbox to 0pt{\hss{#3}\hss}}}
\def\lput(#1,#2)#3{\put(#1,#2){\hbox to 0pt{\hss{#3}}}}

\title[\resizebox{5.3in}{!}{Convergence of subdivision schemes on Riemannian manifolds with nonpositive sectional curvature}]{Convergence of subdivision schemes on Riemannian manifolds with nonpositive sectional curvature}

\author{Svenja H\"uning, Johannes Wallner}

\address{Svenja H\"uning (huening@tugraz.at) and
Johannes Wallner (j.wallner@tugraz.at). Institut f.\ Geometrie, TU Graz. Kopernikusgasse 24, 8010 Graz}

\usepackage{amsmath,amsfonts,amssymb,amsthm,overpic,relsize}
\usepackage[square,numbers]{natbib}

\usepackage[a4paper]{geometry}
\usepackage{float}

\theoremstyle{plain}
\newtheorem{thm}{Theorem}
\newtheorem{lem}[thm]{Lemma}

\newtheorem{prop}[thm]{Proposition}

\theoremstyle{definition}
\newtheorem{Def}[thm]{Definition}
\newtheorem{ex}[thm]{Example}

\theoremstyle{remark}

\def\av{\operatorname{av}}
\def\dist{\operatorname{dist}}
\def\grad{\operatorname{grad}}
\def \<{\langle}
\def \>{\rangle}
\def\N{{\mathbb N}}
\def\Z{{\mathbb Z}}
\def\R{{\mathbb R}}
\def\pds{\partial s}
\def\pdu{\partial u}
\def\Exende{\def\qedsymbol{$\diamondsuit$}\qed}
\def\CH{Car\-tan-\hskip0pt Ha\-da\-mard}

\begin{document}
\long\def\nix#1{}

\maketitle
\begin{abstract}
This paper studies well-defindness and convergence of subdivision schemes which operate on Riemannian manifolds with nonpositive sectional curvature. These schemes are constructed from linear ones by replacing affine averages by the Riemannian center of mass. In contrast to previous work, we consider schemes without any sign restriction on the mask, and our results apply to \textit{all} input data.
We also analyse the H\"older continuity of the resulting limit curves.
Our main result states that convergence is implied by contractivity of a derived scheme, resp.\ iterated derived scheme. In this way we establish that convergence of a linear subdivision scheme is almost equivalent to convergence of its nonlinear manifold counterpart.  
\end{abstract}


\section{Introduction}

Linear stationary subdivision schemes are well-studied regarding their properties of convergence and smoothness, see for example \cite{cavaretta}. Over the last years, linear refinement rules were transferred to nonlinear geometries, and subdivision algorithms have been applied to data coming from surfaces, Lie groups or Riemannian manifolds. Different methods have been introduced to extend linear refinement algorithms to manifold-valued data. Examples are the log-exp-analogue of a linear scheme \cite{donoho,rahman}, geodesic averaging processes or the so-called projection analogue, see \cite{grohs} for an overview. 

Many results on convergence of nonlinear refinement processes are based on the so-called proximity conditions introduced in \cite{wallnerdyn}. These convergence results unfortunately only apply to `dense enough' input data.
 
If convergence is assumed, many nonlinear constructions yield \(C^1\) and \(C^2\) smoothness, see e.g.\ \cite{wallner3, grohs3, wallner4}. The full smoothness of linear schemes is reproduced only if certain ways of constructing nonlinear schemes from linear ones are employed \cite{xie2,grohs}.

Returning to the question of convergence of nonlinear subdivision schemes, some results apply to \textit{all} input data. One can show convergence e.g.\
for interpolatory schemes in Riemannian manifolds
\cite{wallner2} or schemes defined by binary
geodesic averaging \cite{dyn4, dyn3}. If one restricts to special
geometries, more general classes of schemes can be shown to converge
for all input data, e.g.\ schemes with  nonnegative mask in
\CH\ metric spaces have been treated by
\cite{ebner1, ebner2}. In this general setting,
which goes beyond smooth manifolds,
the coefficients of the scheme's mask are interpreted as probabilities. 

In this paper we prove convergence of subdivision schemes in complete Riemannian manifolds with sectional curvature \(K \leqslant 0\). Our results are valid for all input data and for schemes with arbitrary mask. We generalise earlier work, in particular Theorem 5 of \cite{wallner} which can only be applied to schemes with nonnegative mask. To extend linear refinement rules to manifold-valued data we use the Riemannian center of mass \cite{karcher}. Such refinement rules have been investigated by \cite{grohs} regarding their smoothness; and in \cite{wallner} with regard to convergence. A synonym for \lq Riemannian center of mass\rq~ which has been used is \textit{weighted geodesic averaging}.

The paper is organized as follows. First we recall some facts about linear subdivision schemes and their nonlinear counterparts. In particular, we introduce a Riemannian analogue \(T\) of a linear scheme \(S\) and show that it is well-defined in \CH\ manifolds. In Section 4 we prove that \(T\) is contractive and displacement-safe, in the terminology introduced in \citep{dyn3}. Afterwards we deduce our main result which states that if
	\begin{align*}
	\frac1{N^m}\Vert S^{m*}\Vert <1, \qquad \text{for some} \ m=1,2,\ldots,
	\end{align*}
then \(T\) converges to a continuous limit curve. Here \(N\) denotes the dilation factor and \(S^*\) is the derived scheme. Next, we analyse the H\"older regularity of the limit curves. Moreover we describe how to extend our results to a wider class of manifolds by dropping the simple connectivity required for \CH\ manifolds. The last section presents some examples.

\section{Subdivision schemes}
\subsection{Linear subdivision schemes}
A linear subdivision scheme \(S\) maps a sequence of points \((x_i)_{i\in \Z }\) lying in a linear space to a new sequence of points \((Sx_i)_{i\in \Z }\) using the rule
	\begin{align*}
	Sx_i=\sum_{j\in \Z } a_{i-Nj}x_j.
	\end{align*}
Here \(N \in \N \) is the \textit{dilation factor}. We require \(N \geqslant 2\), but the usual case is \(N=2\). Throughout the paper we assume that the sequence \( a_{\ell}\), \(\ell \in \Z \), called the \textit{mask} of the refinement rule, has compact support. This means that \(a_{\ell} \neq 0\) only for finitely many \(\ell\). 
It turns out that the condition 
	\begin{align} \label{affine_invariance}
	\sum_{j\in \Z } a_{i-Nj}=1\ \   \text{for all $i$}
	\end{align}
(\textit{affine invariance}) is necessary for the convergence of linear subdivision schemes, see \cite{dyn3} and \cite{cavaretta} for an overview. From now on, we make the assumption that all subdivision schemes are affine invariant. 

To simplify notation, we initially consider only \textit{binary} refinement rules, i.e., rules with dilation factor \(N=2\). Then we can write the refinement rule in the following way:
	\begin{align} \label{linearrule}
	(Sx)_{2i}
	=\sum_{j=-m}^{m+1}\alpha_j x_{i+j} 
	\hspace{0.5cm} \text{and} \hspace{0.5cm}
	(Sx)_{2i+1}=\sum_{j=-m}^{m+1}\beta_j x_{i+j},
	\end{align}
with \(m\in \N \) and coefficients \(\alpha_j,\beta_j\) such that \begin{align} \label{sumcondition1}
\sum_{j=-m}^{m+1}\alpha_j=\sum_{j=-m}^{m+1}\beta_j=1.
\end{align}
For example Chaikin's algorithm \cite{chaikin}, which is given by the mask
\((a_{-2},\ldots,a_1)$ $=$ $(\frac1{4},\frac 34,\frac 34,\frac 14)\), can be written as 
\begin{align} \label{Chaikin_rule}
	(Sx)_{2i}=\frac{3}{4}x_{i}+\frac1{4}x_{i+1} 
	\hspace{0.5cm} \text{and} \hspace{0.5cm} 
	(Sx)_{2i+1}=\frac1{4}x_{i}+\frac{3}{4}x_{i+1}.
\end{align}
Subdivision schemes satisfying \((Sx)_{2i}=x_{i}\) are called \textit{interpolatory}. For example the well-known four-point scheme is defined by
\begin{align}\label{four-point}
	(Sx)_{2i}=x_i 
	\hspace{0.5cm} \text{and} \hspace{0.5cm} 
	(Sx)_{2i+1}=-\omega x_{i-1}
		+\Big(\frac1{2} +\omega\Big) x_i
		+\Big(\frac1{2}+\omega\Big) x_{i+1}-\omega x_{i+2},
	\end{align}
for some parameter \(\omega\), see \cite{dyn1}. 
The next example will be our main example throughout the text.

\begin{ex} \label{Chaikin_fourpoint}
We consider a non-interpolatory subdivision scheme with negative mask coefficients.
Taking averages of the four-point scheme with parameter \(\omega=\frac1{16}\) and Chaikin's scheme yields 
\setlength{\belowdisplayskip}{0pt}%
\setlength{\belowdisplayshortskip}{0pt}%
\begin{align*} 
	(Sx)_{2i}
	&=-\frac1{32}x_{i-1}
	+\frac{21}{32}x_i
	+\frac{13}{32}x_{i+1}
	-\frac1{32}x_{i+2}, 
	\\
	(Sx)_{2i+1}
	&=-\frac1{32}x_{i-1}
	+\frac{13}{32}x_i
	+\frac{21}{32}x_{i+1}
	-\frac1{32}x_{i+2}.
\end{align*} 
\Exende
\end{ex}

\subsection{The Riemannian analogue of a linear subdivision scheme} \label{Riemannian analogue}

We recall the extension of a linear subdivision scheme to manifold-valued data with the help of the Riemannian center of mass as shown in \cite{grohs}. This generalisation of the concept of affine average is quite natural in the sense that we only replace the Euclidean distance by the Riemannian distance. The construction requires to introduce some notation. We denote the Riemannian inner product by \(\< \cdot,\cdot\>\). The Riemannian distance \(\dist (x,y)\) between two points  \(x,y \in M\) is given by
	\begin{align*}
	\dist (x,y):=\inf\limits_{\gamma} 
	\int_a^b \left|\dot\gamma(t)\right|\, dt,
	\end{align*}  
where \(\gamma:[a ,b]\to M\) is a curve connecting points \(\gamma(a)=x\) and \(\gamma(b)=y\).
Consider the weighted affine average 
\begin{align*}
	x^*=\sum_{j=0}^n \alpha_j x_j
	\end{align*}
 of points \(x_j \in \R ^d\) w.r.t.\ weights \(\alpha_j \in \R \), satisfying \(\sum \alpha_j=1\). It can be characterised as the unique minimum of the function 
	\begin{align*}
	g_{\alpha}(x)=\sum_{j=0}^n \alpha_j \left| x-x_j\right|^2.
	\end{align*}
We transfer this definition to Riemannian manifolds by replacing the Euclidean distance by the Riemannian distance. Let
	\begin{align*}
	f_{\alpha}(x)=\sum_{j=0}^n \alpha_j\dist ( x,x_j)^2.
	\end{align*}
We call the minimizer of this function the \textit{Riemannian center of mass} and denote it by 
	\begin{align*}
	x^*=\av(\alpha,x).
	\end{align*}
Note that in general the Riemannian center of mass is only well-defined locally. It is the aim of the present paper to identify settings where the average is globally well-defined.
We extend the linear subdivision rule (\ref{linearrule}) to manifold-valued data by defining
	\begin{align} \label{i-mean}
	(Tx)_{2i}=\av(\alpha,x) \hspace{0.5cm} \text{and} \hspace{0.5cm}
	(Tx)_{2i+1}=\av(\beta,x).
\end{align}

\begin{Def}
We call \(T\) the \textit{Riemannian analogue} of the linear subdivision scheme \(S\).
\end{Def}

\section{The Riemannian center of mass in \CH\ manifolds}
\label{Rcom}

\noindent
\CH\ manifolds, and more generally manifolds with nonpositive sectional curvature, are a class of geometries where the Riemannian average can be made globally well-defined. Let \(M\) be a \CH\ manifold, i.e., a simply connected, complete Riemannian manifold with sectional curvature \(K \leqslant 0\). To show well-definedness of geodesic averages we have to clarify the global existence and uniqueness of a minimizer of the function 
	\begin{align} \label{minimizationfunction}
f_{\alpha}(x)=\sum_{j=-m}^{m+1} \alpha_j \dist (x_j,x)^2, 
	\quad\text{with}\quad
	\sum_j \alpha_j =1
\end{align}
and \(x_j \in M\). 
A local answer to this question is not difficult, see for example  \cite{sander}. The global well-definedness in case \(\alpha_j \geqslant 0\) is shown in \cite{kobayashi}. 
Hanne Hardering gave another proof of the global existence in \cite{hardering}. We are mainly interested in the result she gave in Lemma \(2.3.\) of \cite{hardering} which we formulate as

\begin{lem} [H. Hardering, \cite{hardering}] \label{existence}
The function \(f_{\alpha}\) has at least one minimum. Moreover, there exists \(r>0\) (depending on the coefficients \(\alpha_j\) and the distances of the points \(x_j\) from each other) such that all minima of \(f_{\alpha}\) lie inside the compact ball \(\overline{B_r (x_0)}\). 
\end{lem}

To prove that the function \(f_{\alpha}\) has a unique minimum we generalise a statement of Hermann Karcher \cite{karcher}. It turns out that we can use arguments similar to his by splitting \(\sum_{j=-m}^{m+1} \alpha_j \dist (x_j,x)^2\) into two sums depending on whether the corresponding coefficient is negative or not. 
Before we introduce the general notation used throughout the text, we illustrate the idea by means of Example \ref{Chaikin_fourpoint}. 
\begin{ex} \label{Chaikin_fourpoint2}
Consider the subdivision rule defined by the coefficients \(\alpha_j\) and \(\beta_j\) of Example \ref{Chaikin_fourpoint} and define \(f_{\alpha}\) according to (\ref{minimizationfunction}) by
	\begin{align*}
	f_{\alpha}(x)=\sum_{j=-1}^2 \alpha_j \dist (x_j,x)^2,
	\end{align*}
with
\((\alpha_{-1},\ldots,\alpha_2)=(-{1\over 32},{21\over 32},
{13\over 32},-{1\over 32})\).
We sort these coefficients in two groups depending on whether they are positive or not. 

It is convenient to define 
	\(
	\alpha_+=\frac{21}{32}+\frac{13}{32}
	=\frac{34}{32}\) 
and 
	\(\alpha_-
	=  \left|-\frac1{32}\right|+\left| -\frac1{32}\right|
	=\frac{2}{32}\). 
 We split the interval 
	\([0,\alpha_{+}+\alpha_{-}]\) 
 in four subintervals whose length coincides with the values 
	\(|\alpha_j|\) 
(but in a different order). We define the function 
	\(\sigma: [ 0,\alpha_{+}+\alpha_{-} ]\to \{-1,0,1,2 \}\) by 
	\begin{align*}
	\sigma(t)=
	\left\{\begin{array}{@{}rl}
	-1
	& \text{for}~ t\in \left[ 0, \frac1{32}\right]\\[1ex]
	2
	& \text{for}~ t\in \left( \frac1{32},\frac{2}{32} \right]\\[1ex]
	0
	& \text{for}~ t\in \left( \frac{2}{32},\frac{23}{32}\right]\\[1ex]
	1
	& \text{for}~ t\in \left( \frac{23}{32},\frac{36}{32}\right]
	\end{array}\right.
	\end{align*}
 and see that
\setlength{\belowdisplayskip}{0pt}%
\setlength{\belowdisplayshortskip}{0pt}%
\begin{align*}
	f_{\alpha}(x)
	&=\sum_{j=-1}^{2}\alpha_{j} \dist (x_{j},x)^2 
	=-\int\limits_0^{\alpha_{-}} 
		\dist (x_{\sigma(t)},x)^2 dt
	+\int\limits_{\alpha_{-}}^{\alpha_{-}+\alpha_{+}} 
		\dist (x_{\sigma(t)},x)^2 dt.
	\qedhere
	\end{align*} \Exende
\end{ex}

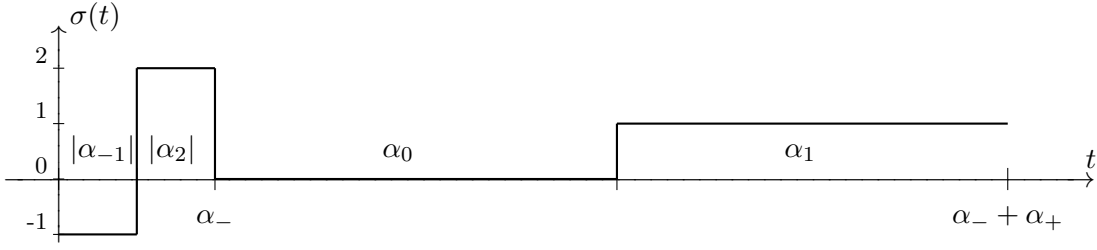
\begin{figure}[t]
	\unitlength0.01\textwidth
	{\begin{picture}(100,20)
	\put(0,4.3){\hbox to 98\unitlength{\rightarrowfill}}
	\put(4.70,-1){\rotatebox{90}{\hbox to 20\unitlength{\rightarrowfill}}}
	\put(5,4){	
		\put(1,15){$\sigma(t)$}
		\put(1,3){$|\alpha_{-1}|$}
		\put(0,0){\line(0,1){2}}
		\thicklines\put(0,-4){\line(1,0){7}}
	}
	\put(12,4){	
		\put(1,3){$|\alpha_{2}|$}
		\put(0,0){\line(0,1){2}}
		\thicklines\put(0,11){\line(1,0){7}}
		\put(0,11){\line(0,-1){15}}
	}
	\put(19,4){	
		\put(15,3){$\alpha_{0}$}
		\put(0,0){\line(0,1){2}}
		\cput(0,-3){$\alpha_-$}
		\thicklines\put(0,1){\line(1,0){36}}
		\put(0,1){\line(0,1){10}}
	}
	\put(55,4){	
		\put(15,3){$\alpha_{1}$}
		\put(0,0){\line(0,1){2}}
		\thicklines\put(0,6){\line(1,0){35}}
		\put(0,6){\line(0,-1){5}}
	}
	\put(90,4){	
		\put(0,0){\line(0,1){2}}
		\cput(0,-3){$\alpha_-+\alpha_+$}	
		\put(7,2){$t$}
	}
	\multiput(4.5,0)(0,5){4}{\line(1,0){1}}
	\put(4,5.5){\footnotesize
		\lput(0,-5){-1}
		\lput(0,0){0}
		\lput(0,5){1}
		\lput(0,10){2}}
	\end{picture}}
	\caption{Construction of the index selection
	function $\sigma$ on basis
		of the sequence $(\alpha_j)_{j=-1}^{2}$ with
	$\alpha_{-1},\alpha_2<0$, $\alpha_0,\alpha_1>0$.}
\end{figure}

In the general case, we need the following
notation to eventually rewrite the function
in (\ref{minimizationfunction}) as the sum of two integrals. We begin to sort our coefficients in two groups by defining two index sets  
	\begin{align*}
	I_{-}^{\alpha}:=\{ j \mid \alpha_{j} < 0 \}, 
	\quad
	I_{+}^{\alpha}:=\{ j \mid \alpha_{j} \geqslant 0 \}.
	\end{align*}
We describe these sets as
	\begin{align*}
	I_{-}^{\alpha}&=\{ j_1,\ldots,j_{n}\}, 
	\quad
	I_{+}^{\alpha}=\{ j_{n+1},\ldots,j_{2m+2}\},  
	\end{align*}
with \(j_1 < \ldots < j_{n}\) and \(j_{n+1} <\ldots<j_{2m+2}\) for
\(n \in \{ 1,\ldots,2m+2\} \) and \(j_{i} \in \{ -m,\ldots,m+1\}\). 
If \(I_{-}^{\alpha}=\emptyset\), we set \(n=0\) and
\(I_{+}^{\alpha}=\{ -m,\ldots,m+1\}\).
Let
\begin{align*}
	\alpha_{+}=\sum_{j \in I_{+}^{\alpha}} \alpha_{j}, \quad
	\alpha_{-}=\sum_{j \in I_{-}^{\alpha}} | \alpha_{j} |, \quad
	\beta_{+}=\sum_{j \in I_{+}^{\beta}} \beta_{j},  \quad
	\beta_{-}=\sum_{j \in I_{-}^{\beta}} | \beta_{j} |.
\end{align*} 
 Assumption (\ref{sumcondition1}) implies that 
	\begin{align} \label{sumcondition2}
	\alpha_{+}-\alpha_{-}=\beta_{+}-\beta_{-}=1.
	\end{align} 
We are now able to rewrite the function \(f_{\alpha}\) with the help of two integrals 
\begin{align} \label{functionasintegral}
	f_{\alpha}(x)&=\sum_{j=-m}^{m+1} 
	\alpha_{j}\ \dist (x_{j},x)^2 
	=
	\Big(-\int_0^{\alpha_{-}} 
	+\int_{\alpha_{-}}^{\alpha_{-}+\alpha_{+}} \Big)
		\dist \left(x_{\sigma(t)},x\right)^2 dt
	\end{align}
with the function \(\sigma:\left[0,\alpha_{+}+\alpha_{-}\right]\to \{ -m,\ldots,m+1 \}\) given as follows. It is constant in each of the successive intervals of length \(| \alpha_{j_1}|,| \alpha_{j_2}|,\dots,| \alpha_{j_{2m+2}}|\) which tile the interval \(\left[0,\alpha_{+}+\alpha_{-}\right]\). Its value in the \(k\)-th interval is given by the integer \(j_{k}\). The values at subinterval boundaries are not relevant.
We note that the first part of the definition of \(\sigma\) represents the summands of (\ref{minimizationfunction}) corresponding to coefficients of \(I^{\alpha}_{-}\) whereas the second part represents the coefficients corresponding to \(I^{\alpha}_{+}\).

Using the representation of the function \(f_{\alpha}\) given in (\ref{functionasintegral}) we can state 

\begin{lem} \label{gradientconvexity}
In a \CH\ manifold the gradient of the function \(f_{\alpha}\) is given by the formula
\begin{align*}
	\frac1{2}\grad f_{\alpha}(x)
	=\int_0^{\alpha_{-}} 
		\exp_x^{-1}x_{\sigma(t)} dt
	-\int_{\alpha_{-}}^{\alpha_{-}+\alpha_{+}} 
		\exp_x^{-1}x_{\sigma(t)} dt,
	\end{align*}
where \(\exp\) denotes the Riemannian exponential map.
Furthermore, we have
	\begin{align*}
	\frac{d^2}{ds^2}f_{\alpha}(\gamma(s))\geqslant \< \dot \gamma (s),
	 \dot \gamma (s) \>
	\end{align*}
for any geodesic \(\gamma:[0,1] \to M\).
\end{lem}

We discuss the proof of this lemma only briefly because the structure and the main ideas are similar to those used in the proof of Theorem \(1.2.\) in \cite{karcher}. 

\begin{proof}
Let \( \gamma: [0,1] \to M\) be a geodesic and denote by 
	\[c_{t}(u,s)
	=\exp_{x_{\sigma(t)}}
		\big(u\cdot \exp_{x_{\sigma(t)}}^{-1}\gamma(s)\big)\] 
the geodesic connecting 
	\(x_{\sigma(t)}\) with \(\gamma(s)\). 
Those geodesics exist and are unique since \(M\) is \CH. Additionally, let 
	\(c'_{t}(u,s):=\frac{d}{du}c_{t}(u,s)\) and 
	\(\dot c_{t}(u,s):=\frac{d}{ds}c_{t}(u,s)\). By 
constrution,
	\( \dist \left(x_{\sigma(t)},\gamma(s)\right)=\|c'_{t}(u,s)\|\). 
 For each \(t,s\) the vector field 
	\(J(u)=\dot c_{t}(u,s)\) 
 along the geodesic 
	\(u\mapsto c_{t}(u,s)\) 
is a Jacobi field. We obtain 
\begin{align*}
	\frac1{2}\frac{d}{ds}f_{\alpha}(\gamma(s))
	&=
	\Big(-\int_0^{\alpha_{-}} 
	+\int_{\alpha_{-}}^{\alpha_{-}+\alpha_{+}} \Big)
		\Big\< \frac{\nabla}{\pds}c'_{t}(u,s),
		\ c'_{t}(u,s)\Big\>\ dt
	\\
	&=
	\Big(-\int_0^{\alpha_{-}} 
		+\int_{\alpha_{-}}^{\alpha_{-}+\alpha_{+}} \Big)
		\Big\< \dot c_{t}(1,s),\ c'_{t}(1,s)\Big\>\ dt,
\end{align*}
where \(\frac{\nabla}{\pds}\) denotes the covariant derivative along the curve \(\gamma(s)\). Observe that \newline \(c'_{t}(1,s)=-\exp^{-1}_{\gamma(s)}x_{\sigma(t)}\) and \(\dot c_{t}(1,s)=\dot \gamma(s)\) is independent of \(t\). We obtain
\begin{align*}
	\frac1{2}\frac{d}{ds}f_{\alpha}(\gamma(s))&
	=\Big\< \dot \gamma(s),\
	\Big(\int_0^{\alpha_{-}} -\int_{\alpha_{-}}^{\alpha_{-}+\alpha_{+}}\Big)
		\exp^{-1}_{\gamma(s)}x_{\sigma(t)} dt
	\Big\>  
\end{align*}
and therefore
\begin{align*}
	\frac1{2} \grad  f_{\alpha}(x)
	=
		\Big(\int_0^{\alpha_{-}} 
		-\int_{\alpha_{-}}^{\alpha_{-}+\alpha_{+}}\Big)
		\exp_x^{-1}x_{\sigma(t)} dt.
\end{align*}
Using (\ref{sumcondition2}) we see that
\begin{align*}
\frac1{2}~\frac{d^2}{ds^2}f_{\alpha}(\gamma(s))&
	=
	\Big(-\int_0^{\alpha_{-}} 
		+\int_{\alpha_{-}}^{\alpha_{-}+\alpha_{+}} \Big)
	\Big\< \dot c_{t}(1,s),\
		 \frac{\nabla}{du}\dot c_{t}(1,s)\Big\> ~dt
	\\
&= 
	\Big(-\int_0^{\alpha_{-}} 
		+\int_{\alpha_{-}}^{\alpha_{-}+\alpha_{+}} \Big)
	\< J(1),J'(1)\big\> ~dt
	\\
	&= \< J(1),J'(1) \> 
	\geqslant \frac1{2} \< \dot \gamma (s) , \dot \gamma (s) \>.
\end{align*} To obtain the inequality above we used the following relations between the Jacobi field and its derivative
\begin{align} \label{inequalitysectionalcurvature}
J'(1)^{\text{tan}}=J(1)^{\text{tan}}\hspace{0.5cm}
\text{and}
\hspace{0.5cm}
\< J' (1)^{\text{norm}},J(1) \> \geqslant \< J (1)^{\text{norm}},J(1) \>,
\end{align}
where \(J^{\text{tan}}\) (resp.\ \(J^{\text{norm}}\)) denotes the tangential (resp.\ normal) part of the Jacobi field; see Appendix A in \cite{karcher} for more details. Here we used the fact that the sectional curvature of \(M\) is bounded above by zero. 
\end{proof}

We sum up the results of the two lemmas above to state the main result of this section.

\begin{thm} \label{intrinsicmean}
In a \CH\ manifold \(M\), the function 
\begin{align*}
f_{\alpha}(x)=\sum_{j=-m}^{m+1} \alpha_{j}\dist (x_{j},x)^2 
\end{align*}
\((\sum \alpha_{j}=1)\) with \(x_{j} \in M\) has a unique minimum. This implies that the geodesic average is globally well-defined in \CH\ manifolds.
\end{thm}

\begin{proof}
By Lemma \ref{existence} there exists a minimum of the function \(f_{\alpha}\) and all its minima are inside a compact ball. By the second part of Lemma \ref{gradientconvexity} the function \(f_{\alpha}\) is strictly convex, so the minimum is unique.
\end{proof}

\section{Convergence result}
In this section we prove that the Riemannian analogue of a linear subdivision scheme in a \CH\ manifold converges for all input data, if the mask satisfies a contractivity condition with contractivity factor smaller than \(1\), see Theorems \ref{convergence} and \ref{derivedscheme}.
The condition implying convergence involves derived schemes (and iterates of derived schemes) and is entirely analogous to a well-known criterion which applies in the linear case. 
This kind of result was previously only known for schemes with nonnegative mask (see Theorem 5 of \cite{wallner}). It has already been conjectured in \cite{grohs}. 

\subsection{Contractivity condition}
We begin by adapting Lemma \(3\) of \cite{wallner}.
\begin{lem} \label{inequality2}
Consider points \(x_{j}\), coefficients \(\alpha_{j}\), \(\beta_{j}\), for \(j=-m,\ldots,m+1\), and their center of mass \(x^*=\av(\alpha,x)\),
\(x^{**}=\av(\beta,x)\) in a \CH\ manifold.
Moreover we assume that (\ref{sumcondition1}) holds. Then
\begin{align*}
	\dist (x^{*},x^{**})
	\leqslant 
	\Big(\sum\limits_{j=-m}^{m+1}\,\Big| \sum\limits_{i\leqslant j}\alpha_{i}-\beta_{i}\Big|\, \Big) \cdot \max\limits_{\ell}
\dist (x_{\ell},x_{\ell+1}).
\end{align*}
\end{lem}
To prove this result we make use of the representation of \(f_{\alpha}\) (resp.\ \(f_{\beta}\)) as in (\ref{functionasintegral}) in terms of the function \(\sigma\) (resp.\ \(\tau\)). 
Before we give the proof of Lemma \ref{inequality2} we illustrate the idea by means of our main example:

\begin{ex} \label{Chaikin_fourpoint4}
From Example \ref{Chaikin_fourpoint2} we know that 
\begin{align*}
	f_{\alpha}(x)
	=
		-\int_0^{\alpha_{-}} 
		\dist (x_{\sigma(t)},x)^2 dt
		+\int_{\alpha_{-}}^{\alpha_{-}+\alpha_{+}} 
		\dist (x_{\sigma(t)},x)^2 dt.
\end{align*}
Similarly we obtain 
\begin{align*}
f_{\beta}(x)=-\int_0^{\beta_{-}} \dist (x_{\tau(t)},x)^2 dt+\int_{\beta_{-}}^{\beta_{-}+\beta_{+}} \dist (x_{\tau(t)},x)^2 dt,
\end{align*} 
with \(\beta_{-}=\frac{2}{32}\), \(\beta_{+}=\frac{34}{32}\) and
\begin{align*}
\tau(t)=
	\left\{\begin{array}{@{}rl}
	-1
	& \text{for}~ t\in [ 0, \frac1{32}]\\[0.7ex]
	2
	& \text{for}~ t\in \left( \frac1{32},\frac{2}{32} \right]\\[0.7ex]
	0
	& \text{for}~ t\in \left( \frac{2}{32},\frac{15}{32}\right]\\[0.7ex]
	1
	& \text{for}~ t\in \left( \frac{15}{32},\frac{36}{32}\right].
	\end{array}\right.
\end{align*}
In order to get the desired result in Lemma \ref{inequality2} we estimate the distance between the gradients of the functions \(f_{\alpha}\) and \(f_{\beta}\) at the point \(x^*=\av(\alpha,x)\) (as explained in more detail in the proof of the Lemma \ref{inequality2}). To be able to do so, we make use of Lemma \ref{gradientconvexity} and convert the resulting four integrals in two by considering differences of the coefficients. In this case, we get
\begin{align*}
	&\Big\| \frac 12 \grad f_{\beta}(x^*)
		-\frac 12 \grad  f_{\alpha}(x^*)\Big\|  
	=\Big\|\sum\limits_{j=-1}^{2} 
		(\alpha_{j}-\beta_{j})
		\exp_{x*}^{-1}x_{j}\Big\|
	\\
	=\,&\Big\| 
		-\int_0^{\frac{8}{32}} 
		\exp_{x*}^{-1}x_{\nu(t+\frac{8}{32})} dt
	+\int_0^{\frac{8}{32}} 
		\exp_{x*}^{-1}x_{\nu(t)} dt\Big \|,
\end{align*}
with 
\begin{align*}
\nu(t)=
	\begin{cases}
	0
	& \text{for}~ t\in [ 0, \frac{8}{32}]\\
	1
	& \text{for}~ t\in \left( \frac{8}{32},\frac{16}{32} \right].
\end{cases}
\end{align*} Note that the construction of the function \(\nu\) is similar to the one of \(\sigma\) in (\ref{functionasintegral}).
\Exende
\end{ex}
We are now ready to give the proof of Lemma \ref{inequality2} which follows the structure in \cite{wallner} and the ideas of \cite{karcher}.

\begin{proof}[Proof of Lemma \ref{inequality2}]
Lemma \ref{gradientconvexity} implies that 
	\begin{align*}
	\frac 12 \grad 
		f_{\alpha}(x)
	=\int_0^{\alpha_{-}} 
		\exp_x^{-1}x_{\sigma(t)} dt
	-\int_{\alpha_{-}}^{\alpha_{-}+\alpha_{+}} 
		\exp_x^{-1}x_{\sigma(t)} dt.
	\end{align*}
We now make use of the ideas of the proof of Theorem \(1.5.\) in \cite{karcher} to obtain a lower bound for the absolute value of the gradient of \(\frac1{2}f_{\alpha}(x)\). 
Let \(\gamma\) denote the geodesic starting from \(x^*\) and ending in \(x\) and let \(c_{t}(u,s)=\exp_{x_{\sigma(t)}}\big(u\cdot \exp_{x_{\sigma(t)}}^{-1}\gamma(s)\big)\) be the family of geodesics from \(x_{\sigma(t)}\) to \(\gamma(s)\). We obtain
	\begin{align*}
	&
	\Big\|\frac 12 \grad 
		f_{\alpha}(\gamma(1))\Big\|
	\cdot
		\Big\|\dot\gamma(1)\Big\|
	\geqslant 
	\int_0^1
		\frac{d}{ds}\Big\<\frac 12 \grad f_{\alpha}(\gamma(s)),\
		\dot\gamma(s)\Big\>\ ds 
	\\ 
	&\qquad
	=\int_0^1\Big \<
		\Big(-\int_0^{\alpha_{-}}
	+\int_{\alpha_{-}}^{\alpha_{-}+\alpha_{+}}\Big)
		\frac{\nabla}{\pdu} \frac{\partial}{\pds}c_{t}(1,s) dt
		,\ \dot c_{t}(1,s) \Big \>\ ds 
	\\
	&\qquad
	=\Big(-\int_0^{\alpha_{-}}
	+\int_{\alpha_-}^{\alpha_-+\alpha_+}\Big)
	\int_0^1
		\Big \<\frac{\nabla}{\pdu} \frac{\partial}{\pds}c_{t}(1,s),
		\ \dot c_{t}(1,s) \Big \>\ ds~dt,
\end{align*}
with \(\dot c_{t}(u,s)=\frac{d}{ds}c_{t}(u,s)\). Let \(J(u)=c_{t}(u,s)\) denote the Jacobi field along the curve \(u\mapsto c_{t}(u,s)\). The dependence on \(s\) and \(t\) is not indicated in the notation. We have \(J(1)=\dot \gamma(s)\) and \(\frac{\nabla}{du} \dot c_{t}(1,s)=J'(1)\).
Using (\ref{sumcondition2}) we obtain
\begin{align*}
	\Big\|\frac 12 \grad 
		f_{\alpha}(\gamma(1))\Big\|
	\cdot
		\Big\|\dot\gamma(1)\Big\|
	&\geqslant 
	\Big(-\int_0^{\alpha_{-}}+
		\int_{\alpha_{-}}^{\alpha_{-}+\alpha_{+}}\Big)
		\int_0^1
	\< J'(1),J(1)\>\ ds~dt
	\\&=
	\< J'(1),J(1)\> 
	\geqslant \< \dot \gamma(s),\dot \gamma(s)\>.
	\end{align*}
The last inequality follows in the same way as in the proof of Lemma \ref{gradientconvexity}. By the definition of the geodesic \(\gamma\) we conclude that
	\begin{align} \label{inequality1}
	\Big\|\frac 12 \grad  f_{\alpha}(x)\Big\|
	\geqslant \dist (x,x^*).
	\end{align}
In a manner similar to \cite[Cor.\ 1.6]{karcher} we observe the following (see also Example \ref{Chaikin_fourpoint4}):
	\begin{align*}
	\Big\| \frac 12 \grad f_{\beta}(x^*)\Big\|
	&=  
	\Big\| \frac 12 \grad  f_{\beta}(x^*)
	- \frac 12 \grad  f_{\alpha}(x^*)\Big\|
	=\Big\| \sum\limits_{j=-m}^{m+1}(\alpha_{j}-\beta_{j})~
		 \exp_{x*}^{-1}x_{j} \Big\|.
	\end{align*}
We define the sequence 
	\(\delta=(\delta_{j})_{j=-m,\dots,m+1}\) by 
	\(\delta_{j}=\alpha_{j}-\beta_{j}\). 
 Let \(\nu\) be the function constructed as \(\sigma\) in (\ref{functionasintegral}) with respect to the coefficients \(\delta\), i.e., the value of \(\nu\) is constant in intervals of length \(| \delta_{j} |\) and given by the corresponding index. Denote by \(\delta_{-}\) (resp.\ \(\delta_{+}\)) the sum of the absolute values of the negative (resp.\ nonnegative) coefficients of \(\delta\). By (\ref{sumcondition1}) \(\delta_{-}=\delta_{+}\). As in (\ref{functionasintegral}) we rewrite the sum above as an integral 
	\begin{align*}
	&\Big\| \sum\limits_{j=-m}^{m+1}(\alpha_{j}-\beta_{j})  
		~\exp_{x*}^{-1}x_{j} \Big\|
	=\Big\|
		-\int_0^{\delta_{-}} 
			\exp_{x*}^{-1}x_{\nu(t)} dt
		+\int_{\delta_{-}}^{\delta_{-}+\delta_{+}} 
			\exp_{x*}^{-1}x_{\nu(t)} dt \Big\| 
	\\
	&=
		\Big\|
		\int_0^{\delta_{-}} 
			\left(-\exp_{x*}^{-1}x_{\nu(t)}
			+\exp_{x*}^{-1}x_{\nu(t+\delta_{-})} \right)
		\,dt \Big\| 
	\leqslant
		\int_0^{\delta_{-}} 
			\Big\|\exp_{x*}^{-1}x_{\nu(t+\delta_{-})}
			-\exp_{x*}^{-1}x_{\nu(t)}  \Big\|
		~dt.
	\end{align*}
With the help of (\ref{inequality1}) we conclude that
\begin{align*}
	& \dist (x^*,x^{**})
	\leqslant \Big\|\frac12 \grad 
		 f_{\beta}(x^*)\Big\| 
	\leqslant 
		\int_0^{\delta_{-}} 
			\Big\|\exp_{x*}^{-1}x_{\nu(t+\delta_{-})}
			-\exp_{x*}^{-1}x_{\nu(t)}  \Big\|~dt 
	\\ &\quad
	\leqslant\int_0^{\delta_{-}} 
			\dist (x_{\nu(t+\delta_{-})},x_{\nu(t)}) dt 
	\leqslant \int_0^{\delta_{-}}
		| \nu(t+\delta_{-})-\nu(t)| ~dt 
		\cdot \max\limits_{\ell}\dist (x_{\ell},x_{\ell+1}).
	\end{align*} 
To obtain the third inequality above we used the fact that in Cartan-Hadamard manifolds,
the exponential map does not decrease distances, see for example \cite{kobayashi}.

It remains to show that 
\begin{align*}
	\int_0^{\delta_{-}}| \nu(t+\delta_{-})-\nu(t)| ~dt
	=\sum\limits_{j=-m}^{m+1}
		\Big| \sum\limits_{i\leqslant j}\alpha_{i}-\beta_{i}\Big|. 
\end{align*}
Therefore we split the sequence of coefficients \(\delta\) in two sequences \(\eta^1\), \(\eta^{2}\) defined by
\setlength{\belowdisplayskip}{0pt}%
\setlength{\belowdisplayshortskip}{0pt}%
\begin{align*}
\eta^1_{j}:=
\begin{cases}
\delta_{j}
& \text{if}~ \delta_{j}\geqslant 0\\
0
& \text{else}\\
\end{cases}
 \hspace{0.5cm} \text{and} \hspace{0.5cm}
 \eta^{2}_{j}:=
\begin{cases}
| \delta_{j}|
& \text{if}~ \delta_{j}< 0\\
0
& \text{else}.\\
\end{cases}
\end{align*}
Similarly to the construction in the proof of Lemma 3 of \cite{wallner} we consider the function \(\epsilon_1\) given by 
\begin{align*}
	\epsilon_1:[ 0,\delta_{-}]&\to \{ -m,\ldots,m+1 \}, \quad
		\epsilon_1(t)
	:=
		\sup\Big\{ j ~\Big|~
		 \sum\limits_{i\leqslant j} \eta^1_{i} < t\Big\} +1.
\end{align*} 
Analogously we define \(\epsilon_{2}\) for the sequence \(\eta^{2}\). We finally obtain
\begin{align*}
	& \int_0^{\delta_{-}}
		| \nu(t+\delta_{-})-\nu(t)| ~dt
	=
	\int_0^{\delta_{-}} 
		| \epsilon_1(t)-\epsilon_2(t)| ~dt 
	\\
	&\quad =
	\sum_{j=-m}^{m+1}
	\Big |\sum_{i\leqslant j} \eta^1_{i}
	-\sum_{i\leqslant j} \eta^{2}_{i} \Big| 
	=
	\sum_{j=-m}^{m+1}
	\Big|\sum_{i\leqslant j}\alpha_{i}-\beta_{i} \Big|.  
\end{align*} 

\smallskip\noindent
 This concludes the proof of Lemma \ref{inequality2}.
\end{proof}

Recall that a binary subdivision scheme \(S\) is given by  \(Sx_{i}\) \(=\) \(\sum_{j\in \Z } a_{i-2j}x_{j}\) with \( \sum_{j\in \Z } a_{i-2j}=1\) for all \(i\). In order to obtain a convergence result for the Riemannian analogue \(T\) of \(S\) we have to estimate the distance between two consecutive points in the sequence \(S^k x\). Let
\(\gamma^{(r)}_{j}= \sum\limits_{i\leqslant j} a_{r-2i}\) and 
\begin{align} \label{contractivityfactor}
	\gamma=
		\max\limits_{r\in \{ 1,2 \}}
		\
		\sum_{j=-m}^{m+1} 
		\left| \gamma^{(r+1)}_{j}-\gamma^{(r)}_{j}\right|.
	\end{align} 
 Then Lemma \ref{inequality2} implies that the subdivision rule \(T\) obeys a so-called \textit{contractivity condition}
\begin{align} \label{contractivity}
	\dist (T^{k}x_{i+1},T^{k}x_{i})
	\leqslant \gamma^{k}\cdot 
	\sup\limits_{\ell} \dist (x_{\ell},x_{\ell+1}).
\end{align}
The factor \(\gamma\) is called \textit{contractivity factor}. In Subsection \ref{secderivedscheme} we show that the value of the contractivity factor \(\gamma\) in (\ref{contractivityfactor}) is closely related to the norm of the derived scheme.

Making use of the result of Hanne Hardering in \cite{hardering} again, it follows that there exists a constant \(C>0\) such that
\begin{align}\label{displacement-safe2}
\dist (Tx_{2i},x_{i}) &\leqslant C\cdot\sup\limits_{\ell} \dist (x_{\ell},x_{\ell+1}), \hspace{0.2cm} i\in \Z .
\end{align}
Subdivision schemes satisfying inequality (\ref{displacement-safe2}) have been called \textit{displacement-safe} by \cite{dyn3}.
Together with (\ref{contractivity}) we conclude that
\begin{align} \label{displacement-safe}
\dist (T^{k+1}x_{2i},T^{k}x_{i}) &\leqslant C \gamma^{k} \varrho \hspace{0.3cm}\text{with} \hspace{0.3cm} \varrho:= \sup\limits_{\ell} \dist (x_{\ell},x_{\ell+1}).
\end{align}

In the linear case (see \cite{dyn2}) a contractivity factor smaller than 1 itself leads to a convergence result, but this condition is not sufficient in the nonlinear case. Here we additionally need the fact that our schemes are displacement-safe as shown in \cite{dyn3} for manifold-valued subdivision schemes based on an averaging process. For interpolatory subdivision schemes however a contractivity factor smaller than 1 entails convergence of the scheme since (\ref{displacement-safe2}) is satisfied anyway, see \cite{dyn3,wallner2}. 

We now state our convergence result which generalises the result of \cite{wallner}.

\begin{thm} \label{convergence}
Consider a binary, affine invariant subdivision scheme \(S\). Denote by \(T\) the Riemannian analogue of \(S\) in a \CH\ manifold \(M\). Let \(\gamma\) be the contractivity factor defined by (\ref{contractivityfactor}). If \(\gamma <1\), then \(T\) converges to a continuous limit \(T^{\infty}x\) for all input data \(x\).
\end{thm}

\begin{proof} [Proof]
Denote by \(c_{k}:\R  \to M\) the broken geodesic which is the union of geodesic segments \(c_{k} \big |{[ \frac{i}{2^k},\frac{i+1}{2^k}]}\) which connect successive points \(T^{k}x_{i}\) and \(T^{k}x_{i+1}\). We show that \(\big(c_{k}|J\big)_{k\geqslant 0}\) is a Cauchy sequence in \(C(J,M)\) for any  interval \(J=[ a,b] \). The metric on \(C(J,M)\) is given by \(\dist (g,h):=\max_{t\in J} \dist (g(t),h(t))\). We now proceed as in the proof of Proposition \(4\) of \cite{wallner}. 
Since \(T\) satisfies (\ref{contractivity}) and is displacement-safe it follows from the definition of the geodesics that
\begin{align*}
\dist (c_{m},c_{m+1})\leqslant \varrho\gamma^m +C\varrho\gamma^m +\varrho\gamma^{m+1}.
\end{align*}
Therefore
\begin{align*}
\dist (c_{m},c_{n})\leqslant 
\big(\varrho +C\varrho +\varrho\gamma\big) \frac{\gamma^m-\gamma^n}{1-\gamma}
\end{align*}
for all \(m \leqslant n\). Thus  \(\big(c_{k}\big|J\big)_{k\geqslant 0}\) is a Cauchy sequence in \(C(J,M)\) for any interval \(J=[ a,b ]\). Completeness of the space \(C(J,M)\) implies existence of the limit function \(T^{\infty}x\).
\end{proof}

\begin{ex}
We compute the contractivity factor of the subdivision scheme introduced in Example \ref{Chaikin_fourpoint}. Using our previous results we get
\begin{align} \label{Chaikin_fourpoint3}
\gamma=\mathrm{max}\Big\{\frac{28}{32},\frac{8}{32}\Big\}=\frac{28}{32}<1.
\end{align}
Thus the Riemannian analogue of the linear scheme converges in \CH\ manifolds for all input data. Figure \ref{chaikinvierpunkt_plot} illustrates the action of this subdivision scheme in the hyperbolic plane.
\Exende
\end{ex}

\begin{figure}[t]
\centering
\includegraphics[scale=1.3]{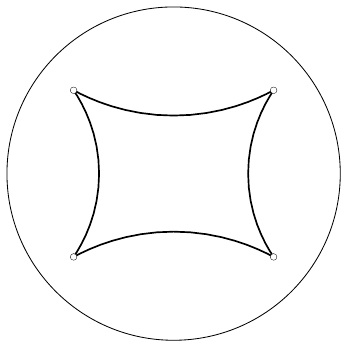}
\hspace*{0.1cm}
\includegraphics[scale=1.3]{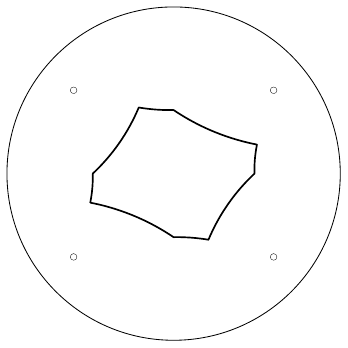}
\hspace*{0.1cm}
\includegraphics[scale=1.3]{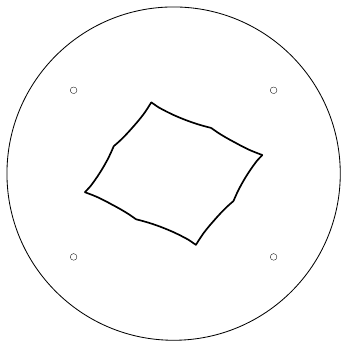}
\caption{Subdivision algorithm of Example \ref{Chaikin_fourpoint} with initial data \(x_0=(0.6, 0.5)\), \(x_1=(0.6, -0.5)\), \(x_2=(-0.6, -0.5)\) and \(x_3=(-0.6 ,0.5)\) in the hyperbolic plane represented with the Poincar\'{e} disk model. From left to right: initial polygon, polygon after one refinement step, polygon after 4 refinement steps.}
\label{chaikinvierpunkt_plot}
\end{figure}

\subsection{Subdivision schemes with dilation factor $N\geqslant2$} \label{arbitrarydilationfactor}

So far we considered subdivision schemes with dilation factor \(N=2\). 
In this section, we extend the convergence result given in Theorem \ref{convergence} to subdivision schemes with arbitrary dilation factor. We still extend a linear subdivision scheme \(S\) to its nonlinear counterpart \(T\) by using the Riemannian analogue introduced in Subsection \ref{Riemannian analogue}. Analogous to the binary case, we say that \(T\) satisfies a contractivity condition with contractivity factor \(\gamma \) if
\begin{align*}
\dist \big(T^{k}x_{i+1},T^{k}x_{i}\big)\leqslant \gamma^{k}\cdot \sup\limits_{\ell} \dist \big(x_{\ell},x_{\ell+1}\big), \hspace{0.2cm} i\in \Z .
\end{align*}
Also we say that \(T\) is displacement-safe if there exists a constant \(C>0\) such that
\begin{align*}
\ \ \dist \big((Tx)_{Ni},x_{i}\big)\leqslant C\cdot\sup\limits_{\ell} \dist \big(x_{\ell},x_{\ell+1}\big), \hspace{0.2cm} i\in \Z .
\end{align*}

\begin{thm} \label{convergencearbirtraryN}
Let \(T\) denote the Riemannian analogue of the linear subdivision rule \(S\) in a \CH\ manifold \(M\) satisfying (\ref{affine_invariance}). Let \(\gamma^{(r)}_{j}= \sum\limits_{i\leqslant j} a_{r-Ni}\) and 
	\begin{align} \label{contractivityfactorN}
		\gamma
	=
		\max \limits_{r\in \{1,\dots,N\}} 
		\sum_{j} 
		\left| \gamma^{(r+1)}_{j}-\gamma^{(r)}_{j}\right|.
	\end{align} If \(\gamma<1\),
then \(T\) converges to a continuous limit \(T^{\infty}x\) for all input data \(x\).
\end{thm}

\begin{proof} [Proof]
As in the binary case, we see that \(T\) satisfies the contractivity condition with factor 
	\[\gamma
	=\max \limits_{r\in \{1,\dots,N\}} \sum_{j}
	 \left| \gamma^{(r+1)}_{j}-\gamma^{(r)}_{j}\right|.\]
 By the result of Hanne Hardering in \cite{hardering} it follows that there exists a constant \(C>0\) such that
\begin{align*}
\dist \big((Tx)_{Ni},x_{i}\big)\leqslant C\cdot \sup\limits_{\ell}
		\dist\big(x_{\ell},x_{\ell+1}\big).
\end{align*}
So \(T\) is displacement-safe. We conclude that
\begin{align*}
\dist \big((T^{k+1}x)_{Ni},(T^{k}x)_{i}\big)\leqslant C\gamma^{k}
	\sup\limits_{\ell}\dist\big(x_{\ell},x_{\ell+1}\big).
\end{align*} 
If we consider the broken geodesics \(c_{k}:\R  \to M\), where \(c_{k} \big |{[ \frac{i}{N^k},\frac{i+1}{N^k}]}\) is the geodesic joining \(T^{k}x_{i}\) and \(T^{k}x_{i+1}\), the claim follows in the same way
as in the proof of Theorem~\ref{convergence}.
\end{proof}

\subsection{Derived scheme} \label{secderivedscheme}
For every linear, affine invariant subdivision scheme \(S\) there exists the \textit{derived scheme} \(S^{*}\) given by the rule \(S^{*}\Delta=N\Delta S\) with \(\Delta x_{i}=x_{i+1}-x_{i}\), see \citep{dyn2}.
In this section we show that the contractivity factor (\ref{contractivityfactorN}) is closely related to the norm 
\begin{align*}
\Vert S^{*} \Vert := \max\limits_{r \in \{ 1,\dots,N\}}\Big\{ \sum\limits_{j} | a^*_{r-Nj} | \Big\}
\end{align*} of the derived scheme \(S^{*}\) with mask \(a^*\). This result is not surprising since it holds in the linear case as well as for nonlinear subdivision schemes with nonnegative mask \cite{wallner}.

\begin{thm} \label{derivedscheme}
Let \(S\) be a linear, affine invariant subdivision rule with dilation factor \(N\). Denote by \(S^{*}\) its derived scheme. If there exists an integer \(m\geqslant 1\) such that \(\frac1{N^m}\Vert S^{m*} \Vert <1\), then the Riemannian analogue \(T\) of \(S\) in a \CH\ manifold converges for all input data. \end{thm}

We can reuse the proof of Theorem 5 of \cite{wallner} to show Theorem \ref{derivedscheme}. We repeat it here for the reader's convenience.

\begin{proof}
Let \(a^*=(a^{*}_{j})_{j \in \Z }\) denote the mask of the derived scheme \(S^{*}\). We consider the special input data \(y=(y_{j})_{j\in \Z }\) given by
\begin{align*}
y_{j}=
	\left\{\begin{array}{@{}rl}
	-1
	& \text{if}~ j\leqslant 0\\[1ex]
	0
	& \text{else}.
	\end{array}\right.
\end{align*} 
We obtain
\begin{align*}
		\frac1{N} a^{*}_ {l}
	&= 
		\frac1{N}\sum\nolimits_{k}a^{*}_{l-Nk}(y_{k+1}-y_{k})
	=
		\frac1{N}S^{*}(y_{l+1}-y_{l})=\frac1{N} S^{*}\Delta y_{l} 
	\\ &
		=\Delta Sy_{l}
		=Sy_{l+1}-Sy_{l}
		=\sum\nolimits_{k\leqslant 0}a_{l-Nk}-a_{l+1-Nk}, \quad\text{and}
	\\ 
		\frac1{N} a^{*}_ {r-Nj}
	&=
		\sum\nolimits_{k\leqslant 0}a_{r-N(j+k)}-a_{r+1-N(j+k)}
	=
		\sum\nolimits_{i\leqslant j}a_{r-Ni}-a_{r+1-Ni}.
\end{align*}
By (\ref{contractivityfactorN}) we get
\begin{align*}
	\sup_r \sum_j |\gamma_{j}^{(r)}-\gamma_{j}^{(r+1)}| 
		=\frac1{N} \sup\limits_{r}\sum\limits_{j} | a^{*}_ {r-Nj} |
		=\frac1{N} \Vert S^{*} \Vert.
	\end{align*}
Since the dilation factor of \(S^m\) is \(N^m\), Theorem \ref{convergencearbirtraryN} gives the desired result. 
\end{proof}

We have just seen that the contractivity factor (\ref{contractivityfactorN}) of the Riemannian analogue of a linear subdivision scheme \(S\) is given by 
\begin{align*}
\gamma=\frac1{N}\Vert S^*\Vert.
\end{align*}
So in order to obtain a convergence result, it suffices to check if the norm of the derived scheme \(S^*\) is smaller than the dilation factor. Even if this is not the case we might get a convergence result by considering iterates of dervied schemes \(S^{m*}\), since the contractivity factor might decrease, see
Subsection \ref{examplefourpointscheme}. 

In \citep{dyn2} it is shown that if we ask for uniform convergence of a linear subdivision scheme \(S\), the existence of an integer \(m\geqslant 1\) such that \(\frac1{N^m}\Vert S^{*m}\Vert<1\) is equivalent to the convergence of the scheme.
Thus Theorem \ref{derivedscheme} states that if the linear subdivision scheme converges uniformly, so does its Riemannian analogue in \CH\ manifolds.

\section{H\"older continuity} 
It has been shown in \cite{wallner2} that the limit function of an interpolatory subdivision scheme for manifold-valued data has H\"older continuity \(-\frac{\log\, \gamma}{\log\, 2}\). Here \(\gamma\) is a contractivity factor for the nonlinear analogue of the linear scheme. It depends only on the mask of the scheme. In \cite{dyn2} a similar inequality is proven for uniformly convergent subdivision schemes in linear spaces. We get the following related result. 
 
\begin{prop} \label{Höldercontinuity}
Let \(T\) be the Riemannian analogue of a binary, affine invariant subdivision scheme \(S\) which has contractivity factor \(\gamma<1\). Then the limit curve \(T^{\infty}x\) satisfies
	\begin{align*}
	\dist \big(T^{\infty}x(t_1), T^{\infty}x(t_{2})\big)
	\leqslant D| t_{2}-t_1|^{\iota},
	\end{align*}
with
\begin{align*}
D=2\cdot\Big(\frac{C\varrho +\varrho + \gamma \varrho}{1-\gamma}+\varrho \Big)\ \ \ \ \ \text{and} \ \ \ \ \ \   \iota=1-\frac{\log \Vert S^*\Vert}{\log 2},
\end{align*}
for all \(t_1, t_{2} \in \R \) with \(| t_1-t_{2} | <1\) and all input data \(x\), i.e., the limit curve is H\"older continuous with exponent $\iota$.

Here the data-dependent constant \(\varrho\) is defined by the maximal distance of successive data points which contribute to the limit curve in the interval under consideration.
\end{prop}

\begin{proof}
Assume that \(t_1, t_{2} \in \R \) with \(|t_1-t_2|<1\). Then there exists an integer \(k \in \Z \) such that \(2^{-k-1}\leqslant |t_2-t_1| \leqslant 2^{-k}\). Together with (\ref{contractivity}) we obtain
	\begin{align*}
		\dist\big(c_{k+1}(t_1),c_{k+1}(t_{2})\big)
	\leqslant 
		2\sup_{\ell} \dist\big (T^{k+1}x_{\ell+1},T^{k+1}x_{\ell}\big)
	\leqslant 	
		2\gamma^{k+1}\varrho.
\end{align*}
Using (\ref{displacement-safe}) we have
\begin{align*}
	& \dist\big (T^{\infty}x(t),c_{k+1}(t)\big)
	\leqslant 
	\lim_{\ell \to \infty}
		 \dist\big (c_{\ell}(t),c_{k+1}(t)\big)
	\\ &\quad \leqslant 
		\sum_{j=k+1}^{\infty} \dist\big (c_{j}(t),c_{j+1}(t)\big)
	= 
	\frac{C\varrho + \varrho+ \gamma\varrho}{1-\gamma}\gamma^{k+1}
\end{align*}
for all \(t \in \R \).
Finally,
\begin{align*}
	&
	\dist\big (T^{\infty}x(t_1),\ T^{\infty}x(t_{2})\big) 
	\\ 
	&
	\leqslant
	\dist\big (T^{\infty}x(t_1),\,c_{k+1}(t_1)\big)
	+\dist\big (c_{k+1}(t_1),\,c_{k+1}(t_{2})\big)
	+\dist\big (c_{k+1}(t_{2}),\,T^{\infty}x(t_{2})\big)
	\\ &
	\leqslant D \gamma^{k+1}
	\leqslant D|t_2-t_1|^{\iota}, 
\end{align*}
with \(\iota=-\frac{\log\, \gamma}{\log\, 2}=1-\frac{\log \Vert S^*\Vert}{\log 2}\).
\end{proof}

\begin{ex}
For our main Example \ref{Chaikin_fourpoint} we compute 
\(\iota=-{\log(\frac{28}{32})}/{\log 2}\approx 0.1926\).
\Exende
\end{ex}

For subdivision schemes with arbitrary dilation factor we obtain
\begin{prop} \label{HoelderarbitraryN}
Let \(T\) be the Riemannian analogue of a linear subdivision scheme \(S\) in a \CH\ manifold \(M\) satisfying (\ref{affine_invariance}). Moreover, we assume that \(T\) has contractivity factor \(\gamma<1\). Then the limit curve \(T^{\infty}x\) satisfies
\begin{align*}
\dist \big(T^{\infty}x(t_1),T^{\infty}x(t_{2})\big)\leqslant D|t_2-t_1|^{\iota},
\end{align*}
with
\begin{align*}
D&=2\cdot \frac{C\varrho +\varrho + (N-1)\gamma \varrho}{1-\gamma}+N\varrho \ \ \ \ \ \text{and} \ \ \ \ \ \   \iota=1-\frac{\log \Vert S^*\Vert}{\log\, N}
\end{align*}
for all \(t_1, t_{2} \in \R \) with \(|t_1-t_2| <1\) and all input data \(x\). Here \(N\) is the dilation factor and the data-dependent constant \(\varrho\) is defined by the maximal distance of successive data points which contribute to the limit curve in the interval under consideration.
\end{prop}

\section{The case of manifolds which are not simply connected}
We explain how to extend our previous results to a complete Riemannian manifold \(M\) with sectional curvature \(K\leqslant 0\), i.e., we drop the assumption of simple connectedness. We use the fact that \(M\) has a so-called simply connected covering (\textit{universal covering}) \(\tilde{M}\). This is a simply connected manifold which projects onto \(M\) in a locally diffeomorphic way. The Riemannian metric on \(M\) is transported to \(\tilde{M}\) by declaring the projection \(\pi:\tilde{M} \to M\) a local isometry. An example is shown by Figure \ref{cylinder}, where
a strip of infinite length and width 1 wraps around the cylinder of height 1 infinitely many times. 
For the general theory of coverings, see e.g.\ \cite{hatcher}. Each data point \(x_j\) in \(M\) has a potentially large number of preimages \(\pi^{-1}(x_j)\).

\begin{figure}[b]
\unitlength0.01\textwidth
{\begin{picture}(100,67)
\cput(25,25){\begin{overpic}[width=.26\textwidth]{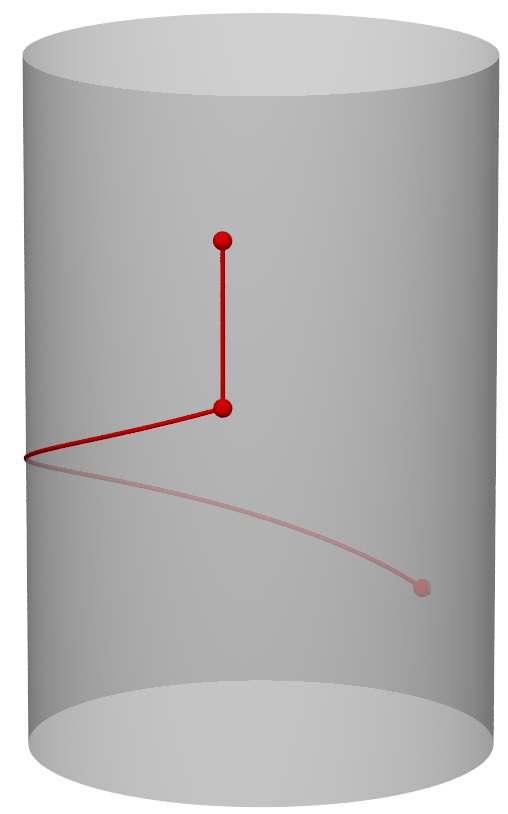}
	\lput(58,78){$M$}
	\put(30,68){$x_0$}
	\put(30,48){$x_1$}
	\put(62,26){$\longleftarrow x_2$}
	\lput(0,42){$c(t)$}
	\end{overpic}}
\cput(75,25){\begin{overpic}[width=.26\textwidth]{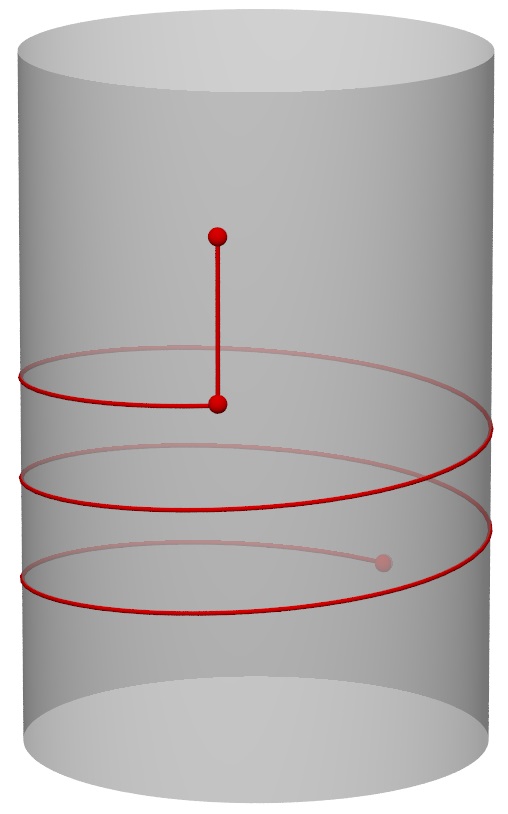} 
	\lput(58,78){$M$}
	\put(30,68){$y_0$}
	\put(30,48){$y_1$}
	\put(62,27){$\longleftarrow y_2$}
	\end{overpic}}
\cput(25,0){\begin{overpic}[width=.48\textwidth]{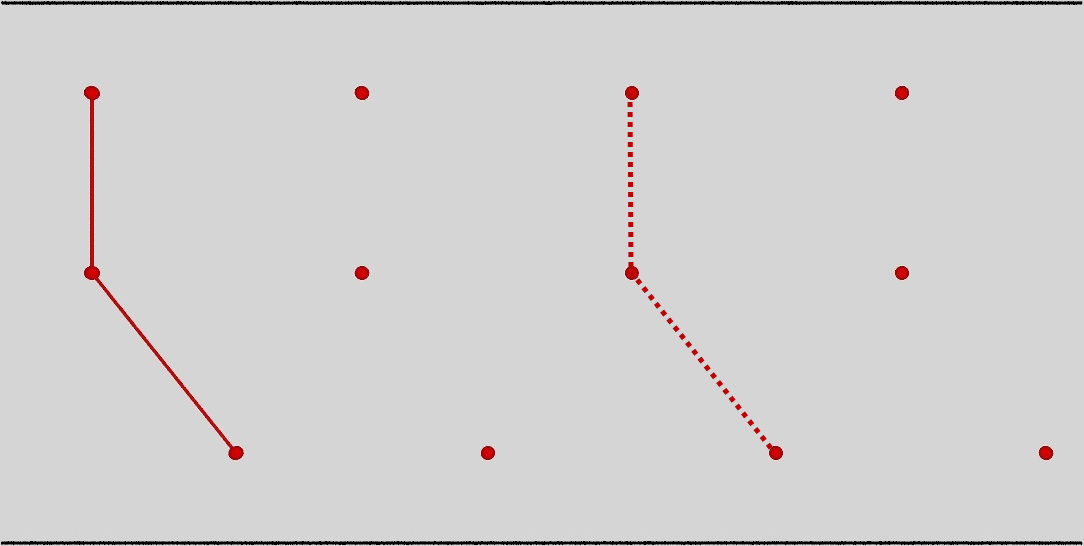}
	\lput(97,44){$\tilde M$}
	\lput(7,38){$\tilde x_0$}
	\lput(7,25){$\tilde x_1$}
	\put(24,7){$\tilde x_2$}
	\put(15,18){$\tilde c(t)$}
	\put(23,13){$\stackrel{\phi}{\hbox to 0.2\textwidth{\rightarrowfill}}$}
	\put(50,0){
		\lput(7,38){$\tilde x_0'$}
		\lput(7,25){$\tilde x_1'$}
		\put(24,7){$\tilde x_2'$}
	}
	\end{overpic}}
\cput(75,0){\begin{overpic}[width=.48\textwidth]{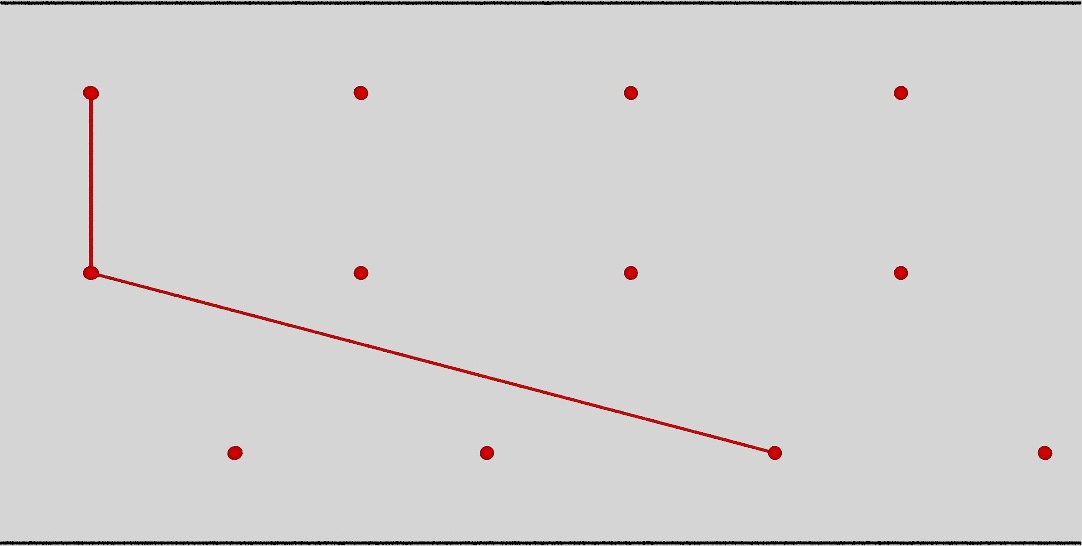}
	\lput(97,44){$\tilde M$}
	\lput(7,38){$\tilde y_0$}
	\lput(7,25){$\tilde y_1$}
	\put(74,7){$\tilde y_2$}
	\end{overpic}}
\end{picture}}

\caption{Top: initial data on a cylinder $M=S^1\times[0,1]$
together with connecting paths. Bottom: their lift to the universal
covering $\tilde M$, which is the strip $(-\infty,\infty)\times[0,1]$.
The various possible liftings are mapped onto each other by a
deck transformation $\phi$.}
\label{cylinder}
\end{figure}
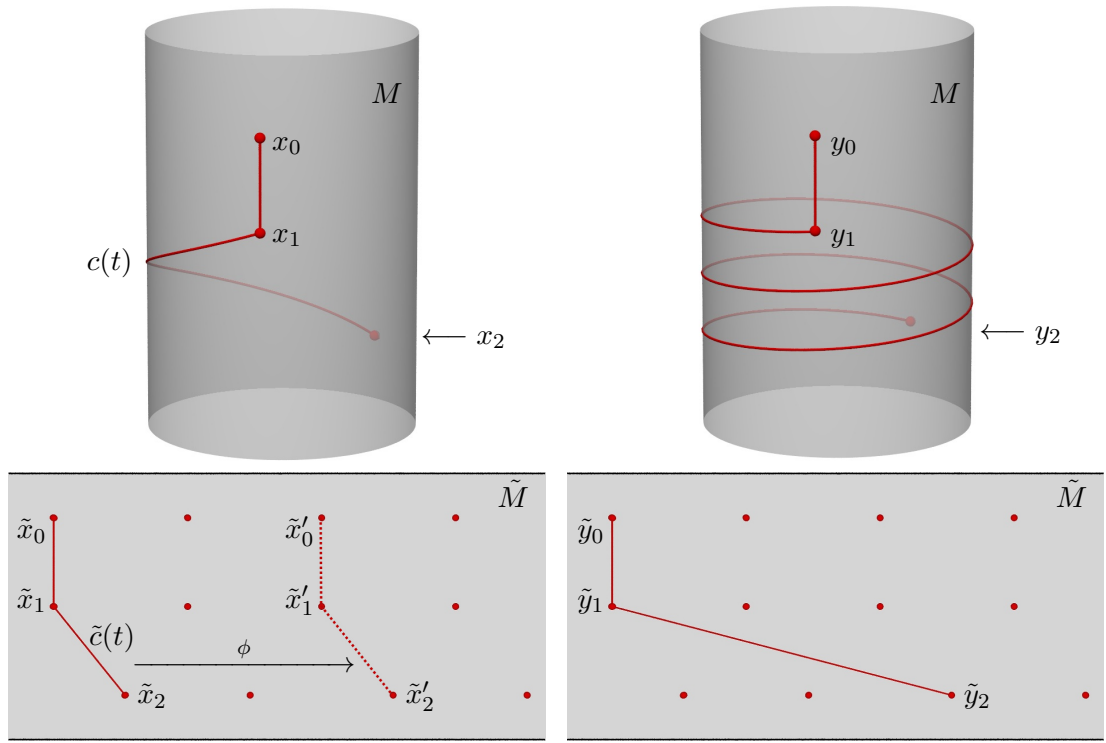

\subsubsection*{Re-definition of the Riemannian analogue of a linear
scheme}

So far our initial data always consisted of a sequence of points in \(M\). Now we additionally choose a path \(c(t)\) which connects the data points \(x_j\) in the correct order: we have \(c(t_j)=x_j\) for suitable parameter values \(\ldots < t_j < t_{j+1} < \ldots\)\ . Such a path is not unique, see Figure \ref{cylinder}. By well-known properties of the simply connected covering, this path can be uniquely lifted to a path \(\tilde{c}(t)\) in \(\tilde{M}\) which projects onto the original path \(c(t)\), once a preimage \(\tilde{x}_0\) with \(\pi(\tilde{x}_0)=x_0\) has been chosen. This means that for all indices \(j\) we have 
\begin{align*}
\tilde{c}(t_j)=\tilde{x}_j,\quad\text{with}\quad\pi(\tilde{x}_j)=x_j.
\end{align*}

We can now simply apply the Riemannian analogue \(\tilde{T}\) of the linear scheme \(S\) which operates on data from \(\tilde{M}\), because \(\tilde{M}\) is \CH\ by construction. Note that there is no Riemannian analogue of \(S\) in \(M\), since \(M\) is not simply connected and geodesic averages are not well-defined in general. However if our input data is a  sequence \(x_j\) together with a connecting path as described above, we may let 
\begin{align*}
Tx=\pi(\tilde{T}\tilde{x}) \quad \text{where \(\tilde{x}\) arrises from \(x\) by lifting}. 
\end{align*}
We can still call \(T\) a natural Riemannian analogue of the linear subdivision scheme \(S\). 

\begin{lem} For any given input data $(x_j)$, the refined data
$(Tx)_j$ computed by the Riemannian analogue $T$ of a linear subdivision
scheme $S$ depends only on the homotopy class of the path $c(t)$ which
is used to connect the data points.
\end{lem}

\begin{proof} First we show that \(Tx\) does not depend on the choice of
the preimage \(\tilde{x}_0\) in the covering space $\tilde M$:
if another preimage \(\tilde{x}'_0\) is chosen,
there is an isometric \textit{deck transformation} \(\phi:\tilde M
\to\tilde M\) which maps the
original lifting to the new one and which commutes with the covering
projection \(\pi\).  The action of \(\tilde{T}\) is invariant under
isometries, so \(\pi(\tilde{T}\tilde{x}')=\pi(\tilde{T}\phi(\tilde{x}))
=\pi(\phi(\tilde{T}\tilde{x}))=\pi(\tilde{T}\tilde{x})\).
Further, it is well known that the lifted location $\tilde x_j$ of
any individual data point $x_j$ depends only on the homotopy class of
the path $c$, cf.\ \cite{hatcher}.
\end{proof} 

With this modification of the notion of input data, our main result Theorem \ref{derivedscheme} now reads as follows:

\begin{thm}
Let \(M\) be a complete manifold with \(K\leqslant0\), and let \(S\) be a linear, affine invariant subdivision rule with dilation factor \(N\). Denote by \(S^{*}\) its derived scheme. If there exists an integer \(m\geqslant 1\) such that \(\frac1{N^m}\Vert S^{m*} \Vert <1\), then the Riemannian analogue \(T\) of \(S\) in \(M\) produces continuous limits for all input data.\end{thm}

\section{Examples}
\subsection{Four-point scheme} \label{examplefourpointscheme}
Consider the general four-point scheme \(S\) introduced in (\ref{four-point}). We would like to know for which values of \(\omega \in (0,\infty)\) the Riemannian analogue \(T\) of \(S\) converges. The mask of the derived scheme is given by \(a^*_{-2}=a^*_{3}=-2\omega\), \(a^*_{-1}=a^*_{2}=2\omega\) and \(a^*_0=a^*_1=1\). Thus by Theorem \ref{convergence} the contractivity factor is \(\gamma=2|\omega|+\frac1{2}\) and \(T\) converges for arbitrary input data if \(-\frac1{4}<\omega< \frac1{4}\). For \(-\frac1{2}<\omega\leqslant 0\) this has already been known \cite{ebner1}, \cite{ebner2}. In this case the mask is nonnegative. 

In particular, we obtain a contractivity factor of 
\(\gamma=\frac{5}{8}\) for the well-studied case of the four-point scheme with \(\omega=\frac1{16}\). By Proposition \ref{Höldercontinuity} we obtain a H\"older exponent of \(\iota\approx 0.6781\). See also Figure \ref{vierpunkt_plot}. 

\begin{figure}[t]
\centering
\includegraphics[scale=1.3]{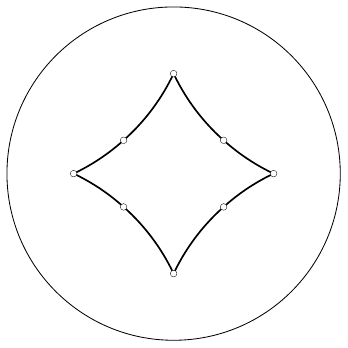}
\hspace*{0.1cm}
\includegraphics[scale=1.3]{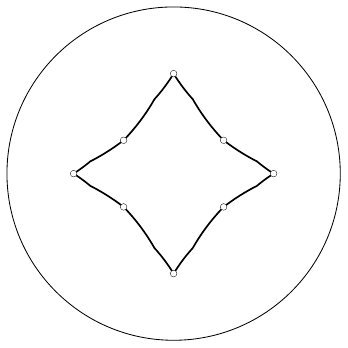}
\hspace*{0.1cm}
\includegraphics[scale=1.3]{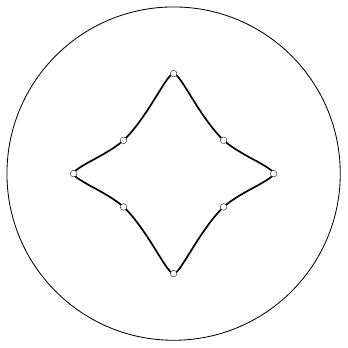}
\caption{Four-point scheme with \(\omega=\frac1{16}\) and initial data \(x_0=(0, 0.6)\), \(x_1=(0.3, 0.2)\), \(x_2=(0.6, 0)\), \(x_3=(0.3, -0.2)\), \(x_4=(0, -0.6)\), \(x_5=(-0.3, -0.2)\), \(x_6=(-0.6 ,0)\) and  \(x_7=(-0.3 ,0.2)\) in the hyperbolic plane represented with the Poincar\'{e} disk model. From left to right: initial polygon, polygon after one refinement step, polygon after 4 refinement steps.}
\label{vierpunkt_plot}
\end{figure}
   
Now we consider two rounds of the four-point scheme  as one round of a subdivision scheme with dilation factor \(N=4\) which for simplicity is again called \(S\). If  \(\omega=\frac1{16}\), our refinement rule is then given by 
\begin{align*}
	\begin{array}{l@{\,\,}*9{@{\,}c@{\,\,}r@{\,}l@{\,}}}
	(Sx)_{4i}
	&=&&&&&&&&x_{i},\\[0.5ex]
	(Sx)_{4i+1}
		&=&\frac1{16^2} \big(& x_{i-2}
		&-&{18} & x_{i-1} 
		&+&{216} & x_{i}
		&+&{66} & x_{i+1}
		&-&{9} & x_{i+2}\big), \\[0.5ex]
	(Sx)_{4i+2}
		&=&\frac1{16^2}\big(&&-&{16} & x_{i-1}
		&+&{144} & x_{i} 
		&+&{144} & x_{i+1}
		&-&{16} & x_{i+2}\big), \\[0.5ex]
	(Sx)_{4i+3} 
		&=&\frac1{16^2}\big(&&-&{9} & x_{i-1}
		&+&{66} & x_{i} 
		&+&{216} & x_{i+1}
		&-&{18} & x_{i+2}
		&+& & x_{i+3}\big). 
	\end{array}
\end{align*}
The contractivity factor is 
\begin{align*}
\gamma=\max \Big\{ \frac{84}{16^2},\frac{80}{16^2} \Big\} = \frac{84}{16^2} \approx 0.3281.
\end{align*}
Theorem \ref{convergencearbirtraryN} again confirms that the Riemannian analogue \(T\) converges to a continuous limit function for all input data.
Proposition \ref{HoelderarbitraryN} yields a H\"older exponent of
\(\iota\approx 0.8039\).
   
\subsection*{Acknowledgements}
The authors acknowledge the support of the Austrian Science Fund (FWF):
This research was supported by the doctoral program {\it Discrete
Mathematics} (grant no.\ W1230) and by the SFB-Transregio 
programme {\it Discretization in geometry and dynamics} (grant no.\ 
I705).

\providecommand{\bysame}{\leavevmode\hbox to3em{\hrulefill}\thinspace}
\providecommand{\MR}{\relax\ifhmode\unskip\space\fi MR }
\providecommand{\MRhref}[2]{%
  \href{http://www.ams.org/mathscinet-getitem?mr=#1}{#2}
}
\providecommand{\href}[2]{#2}


\begin{thebibliography}{10}

\bibitem{cavaretta}
A.~S. Cavaretta, W.~Dahmen, and C.~A. Michelli, \emph{Stationary subdivision},
  Memoirs of the American Mathematical Society, vol.~93, 1991.

\bibitem{chaikin}
G.~M. Chaikin, \emph{An algorithm for high speed curve generation}, Computer
  Graphics and Image Processing \textbf{3} (1974), 346--349.

\bibitem{donoho}
D.~L. Donoho, \emph{Wavelet-type representation of {L}ie-valued data}, talk at
  the IMI ``Approximation and Computation'' meeting, May 12--17, Charleston,
  SC, 2001.

\bibitem{dyn2}
N.~Dyn, \emph{Subdivision schemes in {CAGD}}, Advances in Numerical Analysis
  Vol. II (W.~A. Light, ed.), Oxford Univ. Press, 1992, pp.~36--104.

\bibitem{dyn1}
N.~Dyn, J.~Gregory, and D.~Levin, \emph{A four-point interpolatory subdivision
  scheme for curve design}, Computer Aided Geometric Design \textbf{4} (1987),
  257--268.

\bibitem{dyn4}
N.~Dyn and N.~Sharon, \emph{A global approach to the refinement of manifold
  data}, Mathematics of Computation \textbf{86} (2017), no.~303, 375--395.

\bibitem{dyn3}
\bysame, \emph{Manifold-valued subdivision schemes based on geodesic inductive
  averaging}, Journal of Computational and Applied Mathematics \textbf{311}
  (2017), 54--67.

\bibitem{ebner1}
O.~Ebner, \emph{Convergence of iterative schemes in metric spaces}, Proceedings
  of the American Mathematical Society \textbf{141} (2013), 677--686.

\bibitem{ebner2}
\bysame, \emph{Stochastic aspects of refinement schemes on metric spaces}, SIAM
  Journal of Numerical Analysis \textbf{52} (2014), 717--734.

\bibitem{grohs}
P.~Grohs, \emph{A general proximity analysis of nonlinear subdivision schemes},
  SIAM Journal on Mathematical Analysis \textbf{42} (2010), 729--750.

\bibitem{grohs3}
P.~Grohs and J.~Wallner, \emph{Log-exponential analogues of univariate
  subdivision schemes in {Lie} groups and their smoothness properties},
  Approximation Theory XII: San Antonio 2007 (M.~Neamtu and L.~L. Schumaker,
  eds.), Nashboro Press, 2008, pp.~181--190.

\bibitem{hardering}
H.~Hardering, \emph{Intrinsic discretization error bounds for geodesic finite
  elements}, Ph.D. thesis, FU Berlin, 2015.

\bibitem{hatcher}
A.~Hatcher, \emph{Algebraic topology}, Cambridge University Press, 2002.

\bibitem{karcher}
H.~Karcher, \emph{Riemannian center of mass and mollifier smoothing},
  Communications on Pure and Applied Mathematics \textbf{30} (1977), 509--541.

\bibitem{kobayashi}
S.~Kobayashi and K.~Nomizu, \emph{Foundations of differential geometry},
  vol.~II, Wiley, 1969.

\bibitem{sander}
O.~Sander, \emph{Geodesic finite elements of higher order}, IMA Journal of
  Numerical Analysis \textbf{36} (2016), 238--266.

\bibitem{rahman}
I.~Ur~Rahman, I.~Drori, V.~C. Stodden, D.~L. Donoho, and P.~Schr{\"o}der,
  \emph{Multiscale representations for manifold-valued data}, Multiscale
  Modeling and Simulation \textbf{4} (2005), no.~4, 1201--1232.

\bibitem{wallner3}
J.~Wallner, \emph{Smoothness analysis of subdivision schemes by proximity},
  Constructive Approximation \textbf{24} (2004), 289--318.

\bibitem{wallner2}
\bysame, \emph{On convergent interpolatory subdivision schemes in {Riemannian}
  geometry}, Constructive Approximation \textbf{40} (2014), 473--486.

\bibitem{wallnerdyn}
J.~Wallner and N.~Dyn, \emph{Convergence and ${C}^1$ analysis of subdivision
  schemes on manifolds by proximity}, Computer Aided Geometric Design
  \textbf{22} (2005), 593--622.

\bibitem{wallner4}
J.~Wallner, E.~Nava~Yazdani, and P.~Grohs, \emph{Smoothness properties of {L}ie
  group subdivision schemes}, Multiscale Modeling and Simulation \textbf{6}
  (2007), 493--505.

\bibitem{wallner}
J.~Wallner, E.~Nava~Yazdani, and A.~Weinmann, \emph{Convergence and smothness
  analysis of subdivision rules in {Riemannian} and symmetric spaces}, Advances
  in Computational Mathematics \textbf{34} (2011), 201--218.

\bibitem{xie2}
G.~Xie and T.~P.-Y. Yu, \emph{Smoothness equivalence properties of
  interpolatory {L}ie group subdivision schemes}, IMA Journal of Numerical
  Analysis \textbf{30} (2010), no.~3, 731--750.

\end{thebibliography}
\end{document}